\documentclass[a4paper,12pt,reqno]{amsart}
\usepackage{a4wide}
\usepackage{amsmath}
\usepackage{amssymb}
\usepackage{amsthm}
\usepackage{latexsym}
\usepackage{graphicx}
\usepackage[english]{babel}

\newtheorem{obs} [subsection]{Remark}

\newtheorem{prop}[subsection]{Proposition}
\newtheorem{conj}[subsection]{Conjecture}
\newtheorem{teor}[subsection]{Theorem}
\newtheorem{lema}[subsection]{Lemma}
\newtheorem{cor} [subsection]{Corollary}
\newcommand{\Zng}{$\mathbb Z^n$-graded $S$-module}

\def\sdepth{\operatorname{sdepth}}
\def\qdepth{\operatorname{hdepth}}
\def\hdepth{\operatorname{hdepth}}
\def\depth{\operatorname{depth}}
\def\supp{\operatorname{supp}}
\def\deg{\operatorname{deg}}

\def\dd{\operatorname{d}}

\def\PP{\operatorname{P}}
\begin{document}
\selectlanguage{english}
\frenchspacing

\numberwithin{equation}{section}

\title{On the Hilbert depth of certain monomial ideals and applications}
\author{Silviu B\u al\u anescu$^1$ and Mircea Cimpoea\c s$^2$}
\date{}



\footnotetext[1]{ \emph{Silviu B\u al\u anescu}, University Politehnica of Bucharest, Faculty of
Applied Sciences, 
Bucharest, 060042, E-mail: silviu.balanescu@stud.fsa.upb.ro}
\footnotetext[2]{ \emph{Mircea Cimpoea\c s}, University Politehnica of Bucharest, Faculty of
Applied Sciences, 
Bucharest, 060042, Romania and Simion Stoilow Institute of Mathematics, Research unit 5, P.O.Box 1-764,
Bucharest 014700, Romania, E-mail: mircea.cimpoeas@upb.ro,\;mircea.cimpoeas@imar.ro}

\begin{abstract}
We study the Stanley depth and the Hilbert depth for $I$ and $S/I$, where 
$I\subset S=K[x_1,\ldots,x_n]$ is the intersection of monomial prime ideals with disjoint sets of variables.
As an application, we obtain bounds for the Stanley depth of $I_{n,m}^t$ and $J_{n,m}^t$, 
where $I_{n,m}$ is the $m$-path ideal of the path graph of length $n$ and $J_{n,m}$ is the
the $m$-path ideal of the cycle graph of length $n$.
\end{abstract}

\maketitle

\section*{Introduction}

Let $K$ be a field and $S=K[x_1,\ldots,x_n]$ the polynomial ring over $K$.
Let $M$ be a \Zng. A \emph{Stanley decomposition} of $M$ is a direct sum $\mathcal D: M = \bigoplus_{i=1}^rm_i K[Z_i]$ as a 
$\mathbb Z^n$-graded $K$-vector space, where $m_i\in M$ is homogeneous with respect to $\mathbb Z^n$-grading, 
$Z_i\subset\{x_1,\ldots,x_n\}$ such that $m_i K[Z_i] = \{um_i:\; u\in K[Z_i] \}\subset M$ is a free $K[Z_i]$-submodule of $M$. 
We define $\sdepth(\mathcal D)=\min_{i=1,\ldots,r} |Z_i|$ and $\sdepth(M)=\max\{\sdepth(\mathcal D)|\;\mathcal D$ is 
a Stanley decomposition of $M\}$. The number $\sdepth(M)$ is called the \emph{Stanley depth} of $M$. 

Herzog, Vladoiu and Zheng show in \cite{hvz} that $\sdepth(M)$ can be computed in a finite number of steps if $M=I/J$, 
where $J\subset I\subset S$ are monomial ideals. 
In \cite{apel}, J.\ Apel restated a conjecture firstly given by Stanley in 
\cite{stan}, namely that $\sdepth(M)\geq\depth(M)$ for any \Zng $\;M$. This conjecture proves to be false, in general, for 
$M=S/I$ and $M=J/I$, where $0\neq I\subset J\subset S$ are monomial ideals, see \cite{duval}. For a friendly introduction in the 
thematic of Stanley depth, we refer the reader \cite{her}.

Let $M$ be a finitely generated graded $S$-module. The Hilbert depth of $M$, denoted by $\hdepth(M)$, is the maximal depth of a finitely generated 
graded $S$-module $N$ with the same Hilbert series as $M$. In \cite{lucrare2} we introduced a new method to compute the Hilbert depth of a 
quotient $J/I$ of two squarefree monomial ideals $I\subset J\subset S$; see Section $1$.

In Section $2$ we consider the edge ideal of a complete bipartite graph, that is
$$I:=(x_1,\ldots,x_n)\cap(x_{n+1},\ldots,x_{n+m})\subset S:=K[x_1,\ldots,x_{n+m}],$$
and we study the Stanley depth and the Hilbert depth of $I$ and $S/I$. 

Assume $m\leq n$. In Proposition \ref{p21} we show that
$$m \geq \sdepth(S/I)\geq \min\{m,\left\lceil \frac{n}{2} \right\rceil\}.$$
Also, in Theorem \ref{t21} we prove that
$$\qdepth(S/I)\leq \left\lfloor n+m+\frac{1}{2}-\sqrt{2mn+\frac{1}{4}} \right\rfloor.$$
In particular, we note that $\qdepth(S/I)<m$ if and only if $n\leq 2m-2$.

In Theorem \ref{t22} we prove that 
$$\qdepth(I)=\sdepth(I)=\left\lceil \frac{m}{2} \right\rceil + \left\lceil \frac{n}{2} \right\rceil,$$
if $n$ and $m$ are not both even. Also, we prove that if $n=2s$ and $m=2t$ then
$$t+s\leq \sdepth(I)\leq \qdepth(S/I) = t+s+1.$$
In particular, we have $\qdepth(I)=\left\lfloor \frac{n+m+2}{2} \right\rfloor$ for any $n\geq m\geq 1$.

In Section $3$ we consider a generalization of the ideal from the previous section, namely
$$I:=I_{n_1,\ldots,n_r}:=(x_1,\ldots,x_{n_1})\cap (x_{n_1+1},\ldots,x_{n_1+n_2})\cap \cdots \cap (x_{n_1+\cdots+n_{r-1}+1},\ldots,x_N)\subset S,$$
where $N=n_1+\cdots+n_r$ and $S=K[x_1,\ldots,x_N]$. In Theorem \ref{t31} we prove that
$$ \left\lfloor \frac{N+r}{2} \right\rfloor \geq \qdepth(I) \geq \sdepth(I) \geq \left\lceil \frac{n_1}{2} \right\rceil
+ \cdots + \left\lceil \frac{n_r}{2} \right\rceil.$$
Also, we conjecture that $$\qdepth(I)=\left\lfloor \frac{N+r}{2} \right\rfloor.$$
This formula holds for $r=2$ and if $r\geq 3$ and at most one of the numbers $n_1,\ldots,n_r$ is even.
In Proposition \ref{percolati} we characterize $\hdepth(S/I)$ and $\hdepth(I)$ in combinatorial terms.
In Proposition \ref{metal} we show that
$$\qdepth(S/I)\leq \min\{d\geq r\;:\; \binom{N-d+r-1}{r} < n_1n_2\cdots n_r\}-1.$$

Proposition \ref{betard} yields us to conjecture that
$$\qdepth(S/I) \approx N - \left\lceil \sqrt[r]{r!n_1n_2\cdots n_r} \right\rceil.$$ 

Let $n>m\geq 2$ and $t\geq 1$ be some integers.
In Section $4$ we apply the results from Section $3$ in order to obtain sharper bounds for the Stanley depth
 of $I_{n,m}^t$ and $J_{n,m}^t$, where 
$$I_{n,m}=(x_1x_2\cdots x_m,\;x_{2}x_3\cdots x_{m+1},\;\ldots,x_{n-m+1}\cdots x_n)\subset S:=K[x_1,\ldots,x_n],$$
is the $m$-path ideal associated to path graph of length $n$ and 
$$J_{n,m}=I_{n,m}+(x_{n-m+2}\cdots x_nx_1,\ldots,x_nx_1\cdots x_{m-1})\subset S$$
is the $m$-path ideal associated to the cycle graph of length $n$.

In Theorem \ref{t41} we show that 
$$\sdepth(I_{n,m}^t)\leq \min\{ n-\left\lceil \frac{t_0}{2} \right\rceil, n-\left\lfloor \frac{n-t_0+1}{m+1} \right\rfloor+1\},$$
where $t_0=\min\{t,n-m\}$. In Theorem \ref{t42} we show that
$$\sdepth(J_{n,m}^t)\leq \left\lfloor \frac{n+d}{2} \right\rfloor-1,$$
for any $t\geq n-1$, where $d=\gcd(n,m)$.

\section{Preliminaries}

First, we fix some notations and we recall the main result of \cite{lucrare2}. 

We denote $[n]:=\{1,2,\ldots,n\}$ and $S:=K[x_1,\ldots,x_n]$. 

For a subset $C\subset [n]$, we denote $x_C:=\prod_{j\in C}x_j\in S$.

For two subsets $C\subset D\subset [n]$, we denote $[C,D]:=\{A\subset [n]\;:\;C\subset A\subset D\}$,
      and we call it the \emph{interval} bounded by $C$ and $D$.

Let $I\subset J\subset S$ be two square free monomial ideals. We let:
$$\PP_{J/I}:=\{C\subset [n]\;:\;x_C\in J\setminus I\} \subset 2^{[n]}.$$
A partition of $P_{J/I}$ is a decomposition: $$\mathcal P:\;\PP_{J/I}=\bigcup_{i=1}^r [C_i,D_i],$$
into disjoint intervals.

If $\mathcal P$ is a partition of $\PP_{J/I}$, we let $\sdepth(\mathcal P):=\min_{i=1}^r |D_i|$.
The Stanley depth of $P_{J/I}$ is 
      $$\sdepth(P_{J/I}):=\max\{\sdepth(\mathcal P)\;:\;\mathcal P\text{ is a partition of }\PP_{J/I}\}.$$
Herzog, Vl\u adoiu and Zheng proved in \cite{hvz} that: $$\sdepth(J/I)=\sdepth(\PP_{J/I}).$$
Let $\PP:=\PP_{J/I}$, where $I\subset J\subset S$ are square-free monomial ideals. For any $0\leq k\leq n$, we denote:
$$\PP_k:=\{A\in \PP\;:\;|A|=k\}\text{ and }\alpha_k(J/I)=\alpha_k(\PP)=|\PP_k|.$$
For all $0\leq d\leq n$ and $0\leq k\leq d$, we consider the integers
\begin{equation}\label{betak}
  \beta_k^d(J/I):=\sum_{j=0}^k (-1)^{k-j} \binom{d-j}{k-j} \alpha_j(J/I).
\end{equation}
From \eqref{betak} we can easily deduce that
\begin{equation}\label{alfak}
  \alpha_k(J/I)=\sum_{j=0}^k \binom{d-j}{k-j} \beta_k^d(J/I),\text{ for all }0\leq k\leq d.
\end{equation}
Also, we have that
\begin{equation}\label{betak2}
\beta_k^d(J/I) = \alpha_k(J/I) - \binom{d}{k}\beta_0^d(J/I)-\binom{d-1}{k-1}\beta_1^d(J/I)-\cdots-\binom{d-k+1}{1}\beta_{k-1}^d(J/I).
\end{equation}

\begin{teor}(\cite[Theorem 2.4]{lucrare2})\label{d1}
With the above notations, the \emph{Hilbert depth} of $J/I$ is
$$\qdepth(J/I):=\max\{d\;:\;\beta_k^d(J/I) \geq 0\text{ for all }0\leq k\leq d\}.$$
\end{teor}

As a basic property of the Hilbert depth, we state the following:

\begin{prop}\label{p1}
Let $I\subset J\subset S$ be two square-free monomial ideals. Then $$\sdepth(J/I)\leq \qdepth(J/I).$$
\end{prop}





\section{Edge ideal of a complete bipartite graph}

Let $n$ and $m$ be two positive integers. We let $S=K[x_1,x_2,\ldots,x_{n+m}]$ and we consider 
the square free monomial ideal:
$$I:=(x_1,\ldots,x_n)\cap (x_{n+1},\ldots,x_{n+m}) \subset S.$$
Our aim is to study the Stanley depth and the Hilbert depth of $I$ and $S/I$.

As usual, given a positive integer $k$, we denote $[k]:=\{1,2,\ldots,k\}$.

\begin{obs}\rm
Let $K_{n,m}=(V,E)$ be the complete bipartite graph, that is $V=V'\cup V''$, where 
$V'=\{1,\ldots,n\}$, $V''=\{n+1,\ldots,n+m\}$ and
$E=\{\{i,j\}\;:\;i\in [n],\;j-n\in [m]\}$. Note that 
$I=(x_ix_{n+j}\;:\;i\in [n],\;j\in [m])$ is the edge ideal of $K_{n,m}$.

Also, we mention that $\depth(S/I)=1$, which can be easily checked.
\end{obs}

\begin{prop}\label{p21}
Let $n\geq m\geq 1$ be two integers. Then:
\begin{enumerate}
\item[(1)] $m \geq \sdepth(S/I)\geq \min\{m,\lceil \frac{n}{2} \rceil\}$.
\item[(2)] $m + \lceil \frac{n}{2} \rceil \geq \sdepth(I) \geq \lceil \frac{n}{2} \rceil + \lceil \frac{m}{2} \rceil$.
\item[(3)] If $n\geq 2m-1$ then $\sdepth(S/I)=m$.
\end{enumerate}
 \end{prop}

\begin{proof}
(1) Since $I=I'S\cap I''S$, where $I'=(x_1,\ldots,x_n)\subset S'=K[x_1,\ldots,x_n]$ and \linebreak
$I''=(x_{n+1},\ldots,x_{n+m})\subset S''=K[x_{n+1},\ldots,x_{n+m}]$,
from \cite[Theorem 1.3(2)]{mirci} it follows that
$$\sdepth(S/I'S)\geq \sdepth(S/I)\geq \min\{\sdepth(S/I'S),\;\sdepth_{S''}(S''/I'')+\sdepth_{S'}(I')\}.$$
As $S/I'S\cong S''$, we have that $\sdepth(S/I'S)=m$. 

Also, $S''/I''\cong K$, so
$\sdepth_{S''}(S''/I'')=0$. 

Finally, $\sdepth_{S'}(I')=\lceil \frac{n}{2} \rceil$, see \cite[Theorem 2.2]{biro}.

(2) Since $(I:x_{n+1})=I'S$, from \cite[Proposition 1.3]{pop} (see arXiv version), \cite[Theorem 2.2]{biro}
and \cite[Lemma 3.6]{hvz} we have
$$\sdepth(I) \leq \sdepth(I:x_{n+1}) = \sdepth(I'S) = m+\sdepth_{S'}(I')=m+\left\lceil \frac{n}{2}\right\rceil.$$
The other inequality follows from \cite[Lemma 1.1]{apop} and \cite[Theorem 2.2]{biro}.

(3) If $n\geq 2m-1$ then $\lceil \frac{n}{2} \rceil\geq m$, hence the result follows from (1).
\end{proof}

\begin{lema}\label{alfaa}
Let $n\geq m\geq 1$ be two integers and $N:=n+m$. We have that
\begin{enumerate}
\item[(1)] $\alpha_k(I)=\begin{cases} 0,& 0\leq k\leq 1 \\ \sum_{j=1}^{k-1} \binom{n}{j}\binom{m}{k-j}, &2\leq k\leq N \end{cases}$.
\item[(2)] $\alpha_k(I)=\binom{N}{k}-\binom{n}{k}-\binom{m}{k} + \delta_{k0}$, for all $0\leq k\leq N$.
\item[(3)] $\alpha_k(S/I)=\binom{n}{k}+\binom{m}{k} - \delta_{k0}$, for all $0\leq k\leq N$.
\end{enumerate}
\end{lema}

\begin{proof}
(1) Since $I$ is generated in degree $2$, we have $\alpha_0(I)=\alpha_1(I)=1$. Any squarefree monomial $u\in I$ with $\deg(u)=k\geq 2$
    can be written as $u=v\cdot w$, where $v\in S'=K[x_1,\ldots,x_n]$ and $w\in S''=K[x_{n+1},\ldots,x_N]$ are squarefree monomials. 
		Assume $\deg(v)=j$ with $1\leq j\leq k-1$. Then $\deg(w)=k-j$. Since there are $\binom{n}{j}$ squarefree monomials of degree $j$
		in $S'$ and $\binom{m}{k-j}$ squarefree monomials of degree $k-j$ in $S''$, we easily get the required conclusion.
		
(2)	For $k\leq 1$ the identity can be easily checked. Assume $k\geq 2$. From (1) and the well known combinatorial formula
    $$\sum_{j=0}^k \binom{n}{j}\binom{m}{k-j} = \binom{n+m}{k} =\binom{N}{k},$$
		we get the required conclusion.
		
(3) It follows immediately from (2).		

\end{proof}

\begin{lema}\label{magic}
For any integers $0\leq k\leq d$ and $n\geq 0$ we have that
$$ \sum_{j=0}^k (-1)^{k-j} \binom{d-j}{k-j}\binom{n}{j} = (-1)^k \binom{d-n}{k} = \binom{n-d+k-1}{k}.$$
\end{lema}

\begin{proof}
Using the identity $\binom{x}{k}=\binom{-x+k-1}{k}$ and the Chu–Vandermonde summation, we get
$$ \sum_{j=0}^k (-1)^{k-j} \binom{d-j}{k-j}\binom{n}{j} = \sum_{j=0}^k \binom{-d+k-1}{k-j}\binom{n}{j} =  \binom{n-d+k-1}{k},$$
as required.
\end{proof}

\begin{lema}\label{beta2}
Let $n\geq m\geq 1$ and $0\leq k \leq d\leq N:=n+m$ some integers. We have that
\begin{align*}
&(1)\;\beta_k^d(S/I) = \binom{n-d+k-1}{k}+\binom{m-d+k-1}{k}+(-1)^{k+1}\binom{d}{k},\\
&(2)\; \beta_k^d(I) =\binom{N-d+k-1}{k}-\binom{n-d+k-1}{k}-\binom{m-d+k-1}{k}+(-1)^{k}\binom{d}{k}.
\end{align*}
\end{lema}

\begin{proof}
(1) From \eqref{betak}, Lemma \ref{alfaa}(3) and Lemma \ref{magic} we have that
\begin{align*}
& \beta_k^d(S/I) = \sum_{j=1}^k (-1)^{k-j} \binom{d-j}{k-j} \binom{n}{j} + \sum_{j=1}^k (-1)^{k-j} \binom{d-j}{k-j} \binom{m}{j}
- (-1)^k\binom{d}{k} = \\
& = \binom{n-d+k-1}{k}+\binom{m-d+k-1}{k}+(-1)^{k+1}\binom{d}{k},
\end{align*}
as required.
(2) The proof is similar, using \eqref{betak}, Lemma \ref{alfaa}(2) and Lemma \ref{magic}.
\end{proof}

Note that, if $n\geq 2m-1$ then, according to Proposition \ref{p21}(3) and Proposition \ref{p1} we have
$\qdepth(S/I)\geq \sdepth(S/I) = m.$
Also, $\sdepth(S/I)\leq m$, for any $n\geq m$.

Hence, it is interesting to consider the case $n\leq 2m-2$, in order to find a better upper bound for $\sdepth(S/I)$.

\begin{teor}\label{t21}
Let $n\geq m\geq 1$ be two integers. Then
$$\sdepth(S/I)\leq \qdepth(S/I)\leq \left\lfloor n+m+\frac{1}{2}-\sqrt{2mn+\frac{1}{4}} \right\rfloor.$$
In particular, if $n\leq 2m-2$ then $\qdepth(S/I)<m$.
\end{teor}

\begin{proof}
The first inequality follows from Proposition \ref{p1}.
We consider the quadratic function 
$$\varphi(t)=\frac{1}{2}t(t-1)-(n+m)t + \frac{1}{2}n(n+1)+\frac{1}{2}m(m+1).$$
Note that, according to Lemma \ref{beta2}(1), $\beta_2^d(S/I)=\varphi(d)$.

The roots of $\varphi(t)=0$ are $t_{1,2}=n+m+\frac{1}{2}\pm \sqrt{2mn+\frac{1}{4}}$ and 
therefore 
$$\varphi(t)<0\text{ if and only if }t\in \left(n+m+\frac{1}{2}-\sqrt{2mn+\frac{1}{4}},n+m+\frac{1}{2}+\sqrt{2mn+\frac{1}{4}} \right).$$
From the fact that $\beta_2^d(S/I)=\varphi(d)$ and the above, it follows that 
$$\beta_2^d(S/I)<0\text{ for }d \geq \left\lfloor n+m+\frac{1}{2}-\sqrt{2mn+\frac{1}{4}} \right\rfloor + 1.$$ From Theorem \ref{d1}, we get $\qdepth(S/I)\leq \left\lfloor n+m+\frac{1}{2}-\sqrt{2mn+\frac{1}{4}} \right\rfloor$.

In order to prove the last part, we consider the function $$\psi(x)=x+m+\frac{1}{2}-\sqrt{2mx+\frac{1}{4}},\;x\in [m,\infty).$$
Since $\frac{d\psi}{dx}(x)>0$, $m\leq n\leq 2m-2$ and $\psi(2m-1)=m$, it follows that
$$\left\lfloor 2m+\frac{1}{2}-\sqrt{2m^2+\frac{1}{4}} \right\rfloor \leq \lfloor \psi(n) \rfloor = \left\lfloor n+m+\frac{1}{2}-\sqrt{2mn+\frac{1}{4}} \right\rfloor < \psi(2m-1)=m,$$
as required.
\end{proof}


\begin{teor}
Let $n\geq m\geq 2$ be two integers. Then
$$\qdepth(S/I)=\max\{d\leq n+m\;:\;\binom{d-n}{2\ell}+\binom{d-m}{2\ell} \geq \binom{d}{2\ell}
\text{ for all }1\leq \ell \leq \left\lfloor \frac{d}{2} \right\rfloor \}.$$
Moreover, $\qdepth(S/I)<m$ if $n\leq 2m-2$. Also, $m\leq \qdepth(S/I)\leq n-m+1$ if $n\geq 2m-1$.
\end{teor}

\begin{proof}
Let $q:=\max\{d\leq n+m\;:\;\binom{d-n}{2\ell}+\binom{d-m}{2\ell} \geq \binom{d}{2\ell}
\text{ for all }1\leq \ell \leq \left\lfloor \frac{d}{2} \right\rfloor \}.$

From Lemma \ref{beta2}(1) and the identity $\binom{x}{k}=\binom{-x+k-1}{k}$ it follows that
\begin{equation}\label{beta2l}
\beta_{2\ell}^d(S/I) = \binom{d-n}{2\ell}+\binom{d-m}{2\ell} - \binom{d}{2\ell}.
\end{equation}
Hence $\qdepth(S/I)\leq q$. 

On the other hand, from the proof of Theorem \ref{t21} and \eqref{beta2l}, it follows that
\begin{equation}\label{curcubeu}
q\leq \left\lfloor n+m+\frac{1}{2}-\sqrt{2mn+\frac{1}{4}} \right\rfloor.
\end{equation}
We consider two cases:
\begin{enumerate}
\item[(i)] $n\leq 2m-2$. From Theorem \ref{t21}, it follows that $q<m$. From Lemma \ref{beta2}(2) and $0\leq k\leq q$ with
$k$ odd, we have
$$\beta^q_k(S/I)= \binom{n-q+k-1}{k}+\binom{m-q+k-1}{k}+\binom{q}{k} \geq \binom{n-m+k}{k}+1+\binom{q}{k} > 0.$$
Since, by the definition of $q$, we have $\beta^q_k(S/I)\geq 0$ for all $0\leq k\leq q$ with $k$ even, we conclude
that $\qdepth(S/I)\geq q$. Hence $\qdepth(S/I)=q<m$, as required.
\item[(ii)] $n\geq 2m-1$. First, note that $$m=\sdepth(S/I)\leq \qdepth(S/I)\leq q.$$
From \eqref{curcubeu} and the above it follows that
$$m\leq q\leq \left\lfloor n+m+\frac{1}{2}-\sqrt{2m(2m-1)+\frac{1}{4}} \right\rfloor = \left\lfloor n+m+\frac{1}{2}-(2m-\frac{1}{2})
\right\rfloor = n-m+1.$$
 From Lemma \ref{beta2}(2) and $0\leq k\leq q$ with
$k$ odd, we have
$$\beta^q_k(S/I)= \binom{n-q+k-1}{k} + \binom{q-m}{k} + \binom{q}{k} \geq \binom{m}{k} + 0 + \binom{m}{k} \geq 0.$$
Using the same argument as in the case (i), it follows that $\qdepth(S/I)=q$, as required. 
\end{enumerate}
Thus, the proof is complete.
\end{proof}

\begin{lema}\label{p252}
Let $n\geq m\geq 1$ be two integers. Then
$$\qdepth(I)\leq \left\lfloor \frac{n+m+2}{2} \right \rfloor.$$
\end{lema}

\begin{proof}
If $n+m=2$, that is $n=m=1$, then there is nothing to prove. Assume $n+m\geq 3$.
From Lemma \ref{beta2}(2) and straightforward computations, it follows that 
$$\beta_3^d(I) = \frac{nm(n+m-2d+2)}{2} <0, $$ 
if and only if $d>\frac{n+m+2}{2}$. Hence, we get the required result.
\end{proof}

\begin{teor}\label{t22}
Let $n\geq m\geq 1$ be two integers.
\begin{enumerate}
\item[(1)] If $n$ and $m$ are not both even then we have that:
$$\sdepth(I)=\qdepth(I)=\left\lceil \frac{n}{2} \right\rceil + \left\lceil \frac{m}{2} \right\rceil.$$
\item[(2)] If $n=2t$ and $m=2s$ then we have that:
$$t+s\leq \sdepth(I)\leq \qdepth(I)\leq t+s+1.$$
\end{enumerate}
In both cases, we have $\hdepth(I)=\left\lfloor \frac{n+m+2}{2} \right \rfloor$.
\end{teor}

\begin{proof}
(1) According to Proposition \ref{p21}(2), we have that
\begin{equation}\label{eku1}
 \sdepth(I)\geq \left\lceil \frac{n}{2} \right\rceil + \left\lceil \frac{m}{2} \right\rceil.
\end{equation}
On the other hand, according to Proposition \ref{p1} and Lemma \ref{p252}, we have that
\begin{equation}\label{eku2}
\sdepth(I)\leq \qdepth(I)\leq \left\lfloor \frac{n+m}{2} \right \rfloor + 1.
\end{equation}
Note that, if $n$ and $m$ are not both even, then
\begin{equation}\label{eku3}
\left\lceil \frac{n}{2} \right\rceil + \left\lceil \frac{m}{2} \right\rceil = \left\lfloor \frac{n+m}{2} \right \rfloor + 1.
\end{equation} 
Hence, (1) follows from \eqref{eku1}, \eqref{eku2} and \eqref{eku3}.

(2) From \eqref{eku1} and \eqref{eku2} we have that
$$t+s\leq \sdepth(I)\leq \qdepth(I) \leq t+s+1.$$
On the other hand, for $0\leq k\leq t+s+1$, from Lemma \ref{beta2}(2) we have
\begin{equation}\label{betakts}
\beta_k^{t+s+1}(I) = \binom{t+s-2+k}{k} - \binom{t-s-2+k}{k} - \binom{s-t-2+k}{k} + (-1)^k\binom{t+s+1}{k}.
\end{equation}
By direct computations, from \eqref{betakts} it follows that 
$$\beta_0^{t+s+1}(I)=0,\; \beta_1^{t+s+1}(I)=0,\; \beta_2^{t+s+1}(I)=4st\text{ and }\beta_3^{t+s+1}(I)=0.$$
Also, by straightforward computations, we get
$$\beta_4^{t+s+1}(I)=\beta_5^{t+s+1}(I)=\frac{ts(2s^2+2t^2-1)}{3} > 0.$$
Now, assume $6\leq k \leq t+s+1$. Without any loss of generality, we assume that $t=s+a$, where $a$ is a nonnegative integer.
In order to prove that $\beta:=\beta_k^{t+s+1}(I)\geq 0$, we consider the following cases:
\begin{enumerate}
\item[(i)] $k$ is even. 
\begin{enumerate}
\item[(i.1)] $a=0$. From \eqref{betakts} and the fact that $s\geq 1$ it follows that
$$\beta=\binom{2s-2+k}{k} - \binom{k-2}{k} -\binom{k-2}{k} + \binom{2s+1}{k} \geq \binom{k}{k} + \binom{2s+1}{k} >0.$$
\item[(i.2)] $a=1$. From \eqref{betakts} and the fact that $s\geq 1$ it follows that
$$\beta=\binom{2s-1+k}{k} - \binom{k-1}{k} - \binom{k-3}{k} + \binom{2s+2}{k} \binom{k+1}{k} + \binom{2s+2}{k} >0.$$
\item[(i.3)] $a\geq 2$ and $k\geq a+2$. From \eqref{betakts} we get
$$ \beta = \binom{2s+a+k-2}{k} - \binom{a+k-2}{k} + \binom{2s+a+1}{k} > \binom{2s+a+1}{k} \geq 0.$$
\item[(i.4)] $a\geq 5$ and $k\leq a+1$. From \eqref{betakts}, using $\binom{-x}{k}=(-1)^k\binom{x+k-1}{k}$, we get
\begin{align*}
& \beta = \binom{2s+a+k-2}{k} - \binom{a+k-2}{k} - \binom{a+1}{k}+\binom{2s+a+1}{k} = \\
& = \binom{2s+a+k-2}{k} - \binom{a+k-2}{k} + \binom{2s+a+1}{k} - \binom{a+1}{k} > 0.
\end{align*}
\end{enumerate}
 
\item[(ii)] $k$ is odd.
\begin{enumerate}
\item[(ii.1)] $a=0$. From \eqref{betakts} and the fact that $k\geq 6$ it follows that
$$\beta=\binom{2s-2+k}{k} - \binom{k-2}{k} -\binom{k-2}{k} - \binom{2s+1}{k} \geq \binom{2s+4}{k} - \binom{2s+1}{k} >0.$$
\item[(ii.2)] $a=0$. From \eqref{betakts} and the fact that $k\geq 6$ it follows that
$$\beta=\binom{2s-1+k}{k} - \binom{k-1}{k} -\binom{k-3}{k} - \binom{2s+2}{k} \geq \binom{2s+5}{k} - \binom{2s+2}{k} >0.$$
\item[(ii.3)] $a\geq 2$ and $k\geq a+2$. From \eqref{betakts} and the fact that $k\geq 6$ and $s\geq 1$ we get \small
\begin{align*}
 & \beta = \binom{2s+a+k-2}{k} - \binom{a+k-2}{k} - \binom{2s+a+1}{k} > \binom{2s+a+k-3}{k-1} - \\  
 & - \binom{a+k-2}{k} \geq \binom{a+k-1}{k-1} - \binom{a+k-2}{k} = \frac{ak+k^2-k-a^2+a}{a(a-1)} \binom{a+k-2}{k} > 0.	
\end{align*}
\normalsize
\item[(ii.4)] $a\geq 5$ and $k\leq a+1$. From \eqref{betakts}, using $\binom{-x}{k}=(-1)^k\binom{x+k-1}{k}$, we get
\begin{align*}
& \beta = \binom{2s+a+k-2}{k} - \binom{a+k-2}{k} + \binom{a+1}{k} - \binom{2s+a+1}{k} = \\
& = \binom{2s+a+k-2}{k} - \binom{2s+a+1}{k} - \left( \binom{a+k-2}{k} - \binom{a+1}{k} \right) = \\
& = \sum_{\ell=2s+a+1}^{2s+a+k-3} \binom{\ell}{k-1} - \sum_{\ell=a+1}^{a+k-3} \binom{\ell}{k-1} > 0.
\end{align*}
\end{enumerate}
\end{enumerate}
Hence, $\qdepth(I)=s+t+1$ and the proof is complete.
\end{proof}


\section{A generalization}

Let $n_1,n_2,\ldots,n_r$ be some positive integers, $N=n_1+\cdots+n_r$ and $S:=K[x_1,\ldots,x_N]$. We consider the ideal
$$I:=I_{n_1,n_2,\ldots,n_r}:=(x_1,\ldots,x_{n_1})\cap (x_{n_1+1},\ldots,x_{n_1+n_2})\cap \cdots \cap (x_{n_1+\cdots+n_{r-1}+1},\ldots,x_N)\subset S,$$
which generalize the ideal $I$ from the previous section.

\begin{lema}\label{l31}
With the above notations, we have that
$$\alpha_k(I)=\begin{cases} 0,& k\leq r-1 \\ \sum\limits_{\substack{\ell_1,\ell_2,\ldots,\ell_r\geq 1 \\ \ell_1+\cdots+\ell_r=k} }
  \binom{n_1}{\ell_1}\binom{n_2}{\ell_2}\cdots \binom{n_r}{\ell_r},& k\geq r \end{cases}.$$
\end{lema}

\begin{proof}
Since $I$ is generated by square free monomials of degree $k$, it is clear that $\alpha_k(I)=0$ for $k\leq r-1$.
Assume $k\geq r$ and let $u\in I$ a square free monomial of degree $k$. It follows that $u=u_1\cdots u_r$, where
$$1\neq u_j \in I_j:=(x_{n_1+\cdots+n_{j-1}+1},\ldots,x_{n_1+\cdots+n_j})\text{ for all }1\leq j\leq r.$$
Let $\ell_j=\deg(u_j)\geq 1$. Since there are $\binom{n_j}{\ell_j}$ squarefree monomials of degree $\ell_j$ in $I_j$,
we get the required conclusion.
\end{proof}

\begin{lema}\label{l32}
With the above notations, we have that
\begin{enumerate}
\item[(1)] $\alpha_k(I)=0$ for $k\leq r-1$.
\item[(2)] $\alpha_r(I)=n_1n_2\cdots n_r$.
\item[(3)] $\alpha_{r+1}(I)=n_1n_2\cdots n_r\cdot \frac{n_1+\cdots+n_r-r}{2}$.
\item[(4)] $\alpha_k(I)=\binom{N}{k}$ for $k\geq N-\min\limits_{i=1}^r{n_i}+1$.
\end{enumerate}
\end{lema}

\begin{proof}
(1), (2) and (3) follow immediately from Lemma \ref{l31}. In order to prove (4), it is enough to observe
that any squarefree monomial $u\in S$ of degree $k\geq N-\min\limits_{i=1}^r{n_i}+1$, belongs to $I$.
\end{proof}

\begin{teor}\label{t31}
With the above notations, we have that:
$$ \left\lfloor \frac{N+r}{2} \right\rfloor \geq \qdepth(I) \geq \sdepth(I) \geq \left\lceil \frac{n_1}{2} \right\rceil
+ \cdots + \left\lceil \frac{n_r}{2} \right\rceil.$$
\end{teor}

\begin{proof}
In order to prove the first inequality, let $d>\frac{n_1+\cdots+n_r+r}{2}$ be an integer. From Theorem \ref{d1} and
Lemma \ref{l32}(1,2,3) it follows that
\begin{align*}
& \beta_k^d(I)=0\text{ for }0\leq k\leq r-1,\;\beta_r^d(I)=\alpha_r(I)=n_1n_2\cdots n_r\text{ and }\\
& \beta_{r+1}^d(I) = \alpha_{r+1}(I)-(d-r)\alpha_r(I) = n_1n_2\cdots n_r\left(\frac{n_1+\cdots+n_r+r}{2}-d\right) < 0.
\end{align*}
Hence $\qdepth(I)\leq \frac{n_1+\cdots+n_r+r}{2}$.

The inequality $\qdepth(I) \geq \sdepth(I)$ follows from Proposition \ref{p1}, and the last inequality
follows from \cite[Corollary 1.9(1)]{mirci} and \cite[Theorem 2.2]{biro}.
\end{proof}

Based on our computer experiments, we propose the following conjecture:

\begin{conj}\label{conju}
With the above notations, we have
$$\qdepth(I)=\left\lfloor \frac{N+r}{2} \right\rfloor.$$
\end{conj}

Note that, according to Theorem \ref{t22}, Conjecture \ref{conju} holds for $r=2$.
Also, according to Theorem \ref{t31}, Conjecture \ref{conju} is true when at most one 
of the numbers $n_1,\ldots,n_r$ is even.



\begin{prop}\label{proo}
With the above notations, we have that 
$$N-\min\limits_{i=1}^r{n_i}\geq \qdepth(S/I)\geq \sdepth(S/I)\geq \left\lceil \frac{n_1}{2} \right\rceil + \cdots + \left\lceil \frac{n_r}{2} \right\rceil
 - \min\limits_{i=1}^r \left\lceil \frac{n_i}{2} \right\rceil.$$
\end{prop}

\begin{proof}
From Lemma \ref{l32}(4) it follows that 
$$\alpha_k(S/I)=0\text{ for all }k\geq N-\max\limits_{i=1}^r{n_i}+1.$$
Hence, from \cite[Lemma 1.3]{lucrare2}, we get the first inequality.

The second inequality follows from \cite[Corollary 1.9]{mirci} and \cite[Theorem 2.2]{biro}.
\end{proof}

Note that, Proposition \ref{proo} reproves the fact that $\qdepth(S/I)\geq \depth(S/I)=r-1$.

\begin{lema}\label{l36}
For any $0\leq k\leq N$, we have that:
\begin{enumerate}
\item[(1)] $\alpha_k(I)=\sum_{J\subseteq [r]} (-1)^{|J|} \binom{N-\sum_{i\in J}n_i}{k}$.
\item[(2)] $\alpha_k(S/I)=\sum_{\emptyset \neq J\subseteq [r]} (-1)^{|J|+1} \binom{N-\sum_{i\in J}n_i}{k}$.
\end{enumerate}
\end{lema}

\begin{proof}
(1) For all $1\leq i\leq r$ we let:
$$A_i=\{(\ell_1,\ldots,\ell_r)\;:\;\ell_1+\cdots+\ell_r=k,\ell_i=0\text{ and }\ell_j\geq 0\text{ for }j\neq i\}.$$
Also, we consider the set:
$$A=\{(\ell_1,\ldots,\ell_r)\;:\;\ell_1+\cdots+\ell_r=k\text{ and }\ell_i\geq 0\text{ for all }1\leq i\leq r\}.$$
For any nonempty subset $J\subset [r]$, we let $A_J=\bigcup_{i\in J}A_i$. Also, we denote $A_{\emptyset}=A$.
From Lemma \ref{l31} it follows that 
\begin{equation}\label{ciocio}
\alpha_k(I)=\sum_{(\ell_1,\ldots,\ell_r)\in A_{\emptyset}\setminus(\bigcup_{i=1}^r A_i)} \binom{n_1}{\ell_1}\cdots \binom{n_r}{\ell_r}.
\end{equation}
Note that this equality holds also for $k<r$ as both terms are zero in this case. It is well known that
$$\sum_{(\ell_1,\ldots,\ell_r)\in A} \binom{n_1}{\ell_1}\cdots \binom{n_r}{\ell_r} = \binom{n_1+\cdots+n_r}{k}=\binom{N}{k}.$$
Similarly, if $J\subset [r]$ then
\begin{equation}\label{ciocioc}
\sum_{(\ell_1,\ldots,\ell_r)\in \bigcap_{i\in J}A_i} \binom{n_1}{\ell_1}\cdots \binom{n_r}{\ell_r} = \binom{N-\sum_{i\in J}n_i}{k}.
\end{equation}
From \eqref{ciocio} and \eqref{ciocioc}, using the inclusion exclusion principle, we get the required conclusion.

(2) Follows from (1) and the fact that $\alpha_k(S/I)=\binom{N}{k}-\alpha_k(I)$.
\end{proof}

Note that Lemma \ref{l36} generalizes Lemma \ref{alfaa}(2,3).
From Lemma \ref{l36} and Lemma \ref{magic} we get the following generalization of Lemma \ref{beta2}:

\begin{lema}\label{c37}
For any $0\leq k\leq d \leq N$, we have that:
\begin{enumerate}
\item[(1)] $\beta_k^d(I)=\sum\limits_{J\subseteq [r]} (-1)^{|J|} \binom{N-\sum_{i\in J}n_i-d+k-1}{k}$.
\item[(2)] $\beta_k^d(S/I)=\sum\limits_{\emptyset \neq J\subseteq [r]} (-1)^{|J|+1} \binom{N-\sum_{i\in J}n_i-d+k-1}{k}$.
\end{enumerate}
\end{lema}

\begin{prop}\label{percolati}
With the above notations, we have that:
\begin{align*}
(1)\; \qdepth(I)=\max\{ & d\;:\;  \left\lceil \frac{n_1}{2} \right\rceil +
 \cdots + \left\lceil \frac{n_r}{2} \right\rceil \leq  d\leq \left\lfloor \frac{N+r}{2} \right\rfloor\text{ and} \\
& \sum\limits_{J\subseteq [r]} (-1)^{|J|} \binom{N-\sum_{i\in J}n_i-d+k-1}{k}\geq 0\text{ for all } r \leq k\leq d \}. \\
(2)\;\qdepth(S/I)=\max\{ & d\;:\; \left\lceil \frac{n_1}{2} \right\rceil + \cdots + \left\lceil \frac{n_r}{2} \right\rceil
 - \min\limits_{i=1}^r \left\lceil \frac{n_i}{2} \right\rceil \leq d \leq N - \min\limits_{i=1}^r{n_i} \text{ and} \\
& \sum\limits_{\emptyset \neq J\subseteq [r]} (-1)^{|J|+1} \binom{N-\sum_{i\in J}n_i-d+k-1}{k}\geq 0\text{ for all } r \leq k\leq d \}.
\end{align*}
\end{prop}

\begin{proof}
(1) It follows from Theorem \ref{t31} and Lemma \ref{c37}(1).

(1) It follows from Proposition \ref{proo} and Lemma \ref{c37}(2).
\end{proof}

\begin{lema}\label{liema}
Let $d\geq r$. We have that $$\beta^d_r(S/I)=\binom{N-d+r-1}{r}-n_1n_2\cdots n_r.$$
\end{lema}

\begin{proof}
From Lemma \ref{l32} it follows that
$$\alpha_k(S/I)=\binom{N}{k}\text{ for }k\leq r-1,\;\alpha_{r}(S/I)=\binom{N}{r}-n_1n_2\cdots n_r.$$
Hence, the required result follows from \eqref{betak} and Lemma \ref{magic}.
\end{proof}

\begin{prop}\label{metal}
With the above notations, we have that:
$$\qdepth(S/I)\leq \min\{d\geq r\;:\;\binom{N-d+r-1}{r} < n_1n_2\cdots n_r\}-1.$$
\end{prop}

\begin{proof}
First of all, note that, according to \eqref{betak2}, we have
$$\beta_0^N(S/I)=1\text{ and }\beta_k^N(S/I)=0\text{ for all }1\leq k\leq r-1.$$
Moreover, according to Lemma \ref{liema}, \eqref{betak} and \eqref{betak2}, we have
$$\beta_r^N(S/I) = \binom{N-N+r-1}{r} - n_1n_2\cdots n_r = -n_1n_2\cdots n_r<0.$$
Therefore, we have that $$q:=\min\{d\geq r\;:\;\binom{N-d+r-1}{r} < n_1n_2\cdots n_r\}$$
is well defined and $q\leq N$. Now, it is enough to notice that, from above, 
$\beta_r^q(S/I)<0$ and thus $\qdepth(S/I)\leq q-1$, as required.
\end{proof}

\begin{lema}\label{lamaie}
We have that $$\binom{N-d+r-1}{r}\geq n_1n_2\cdots n_r
\text{ for all }d\leq N - \left\lceil \sqrt[r]{r!n_1n_2\cdots n_r} \right\rceil.$$
\end{lema}

\begin{proof}
We assume that $d=\lfloor aN \rfloor $, where $a\in (0,1)$. Then
$$\binom{N-d+r-1}{r}=\frac{(N-d+r-1)(N-d+r-2)\cdots (N-d)}{r!}\geq $$
\begin{equation}\label{caca1}
 \geq \frac{(N-aN+r-1)(N-aN+r-2)\cdots(N-aN)}{r!} \geq \frac{N^r(1-a)^r}{r!}.
\end{equation}
On the other hand
\begin{equation}\label{caca2}
 \frac{N^r(1-a)^r}{r!} \geq n_1n_2\cdots n_r \Leftrightarrow (1-a)^r \geq \frac{r!n_1n_2\cdots n_r}{r!N^k} \Leftrightarrow 
a\leq 1-\frac{\sqrt[r]{r!n_1n_2\cdots n_r}}{N}
\end{equation}
The conclusion follows from \eqref{caca1} and \eqref{caca2}.
\end{proof}

\begin{obs}\rm
Note that, from the inequality of means, we have 
$$\sqrt[r]{r!n_1n_2\cdots n_r} \leq \frac{N\sqrt[r]{r!}}{r},$$ with equality for $n_1=n_2=\cdots=n_r=n=\frac{N}{r}$.
Therefore, from Lemma \ref{lamaie}, we have that
$$\binom{N-d+r-1}{r}\geq n^r\text{ for all }d \leq N\left( 1 - \sqrt[r]{\frac{r!}{r^r}}\right).$$
\end{obs}

\begin{prop}\label{betard}
With the above notations, we have that
$$\beta_r^d(S/I)\geq 0 \text{ for all }d\leq N - \left\lceil \sqrt[r]{r!n_1n_2\cdots n_r} \right\rceil.$$
\end{prop}

\begin{proof}
It follows from Lemma \ref{liema} and Lemma \ref{lamaie}.
\end{proof}

Proposition \ref{betard} allows us to conjecture that
$$\qdepth(S/I) \approx N - \left\lceil \sqrt[r]{r!n_1n_2\cdots n_r} \right\rceil.$$

\section{Applications}

\subsection*{The $m$-path ideal of a path graph}

Let $n\geq m\geq 1$ be two integers and 
$$I_{n,m}=(x_1x_2\cdots x_m,\;x_{2}x_3\cdots x_{m+1},\;\ldots,x_{n-m+1}\cdots x_n)\subset S=K[x_1,\ldots,x_n],$$
be the $m$-path ideal associated to the path graph of length $n$. Let $t\geq 1$. We define:
$$\varphi(n,m,t):=\begin{cases} n-t+2 - \left\lfloor \frac{n-t+2}{m+1} \right\rfloor - \left\lceil \frac{n-t+2}{m+1} \right\rceil,& t\leq n+1-m \\
m-1,& t > n+1-m \end{cases}.$$
According to \cite[Theorem 2.6]{lucrare1} we have that 
$$\sdepth(S/I_{n,m}^t)\geq\depth(S/I_{n,m}^t)=\varphi(n,m,t).$$
Assume that $t\leq n-m$ and let $S_{t+m}:=K[x_1,\ldots,x_{m+t}]$. We consider the ideal
$$U_{m,t}= (x_{i_1}\cdots x_{i_m}\;:\;i_j\equiv j(\bmod\; m),1\leq j\leq m) \subset S_{t+m}.$$
By Euclidean division, we write $t+m=am+b$, where $1\leq b\leq m$. According to the proof of \cite[Lemma 2.4]{lucrare1},
we have that
\begin{equation}
U_{m,t}=V_{m,1,a+1}\cap \cdots \cap V_{m,b,a+1} \cap V_{m,b+1,a} \cap \cdots \cap V_{m,m,a},\text{ where }V_{m,j,k}=(x_j,x_{j+m},\ldots,x_{j+(k-1)m}).
\end{equation}

\begin{prop}\label{p41}
With the above notations, we have that:
$$\sdepth(U_{m,t})\leq \qdepth(U_{m,t})\leq m+\left\lfloor \frac{t}{2} \right\rfloor.$$
\end{prop}

\begin{proof}
According to Theorem \ref{t31}, we have that
$$\qdepth(U_{m,t})\leq \left\lfloor \frac{m+t+m}{2} \right\rfloor = m + \left\lfloor \frac{t}{2} \right\rfloor.$$
Now, apply Proposition \ref{p1}.
\end{proof}



We recall the following well known results:

\begin{lema}\label{lima}
Let $I\subset S$ be a monomial ideal and $u\in S$ a monomial which do no belong to $I$. Then
\begin{enumerate}
\item[(1)] $\sdepth(S/(I:u))\geq \sdepth(S/I)$. (\cite[Proposition 2.7(2)]{mirci})
\item[(2)] $\sdepth(I:u)\geq \sdepth(I)$. (\cite[Proposition 1.3]{pop} arXiv version)
\end{enumerate}
\end{lema}

\begin{teor}\label{t41}
Let $n\geq m\geq 1$ and $t\geq 1$. Let $t_0:=\min\{t,n-m\}$. We have that
$$\sdepth(I_{n,m}^t)\leq \min\{ n-\left\lceil \frac{t_0}{2} \right\rceil, n-\left\lfloor \frac{n-t_0+1}{m+1} \right\rfloor+1\}.$$
\end{teor}

\begin{proof}

If $t\geq n-m+1$, then $t_0=n-m$ and, according to \cite[Lemma 2.1]{lucrare1}, we have 
$$I_{n,m}^{t_0}=I_{n,m}^t:(x_{n-m+1}\cdots x_n)^{t-t_0}.$$
Therefore, from Lemma \ref{lima}(2) it follows that $$\sdepth(I_{n,m}^t)\leq \sdepth(I_{n,m}^{t_0}).$$
Hence, we can assume that $t\leq n-m$ and $t_0=t$.

By Euclidean division, we write $n-t+1=q(m+1)+r$, where $0\leq r\leq m$. According to
\cite[Lemma 2.5]{lucrare1}, there exists a monomial $w\in S$ such that:
\begin{equation}\label{ciuciu}
(I_{n,m}^t:w)=\begin{cases} U_{m,t}+P_{m,t,q},& r<m \\ U_{m,t}+P_{m,t,q}+(x_{n-m+1}\cdots x_n),& r=m \end{cases},
\end{equation}
where $P_{m,t,q}\subset K[x_{t+m+1},\ldots,x_{t+q(m+1)-1}]$ is a prime monomial ideal of height $2(q-1)$. 

If $r<m$ then, from \eqref{ciuciu} 
and \cite[Theorem 1.3]{mirci} it follows that 
$$\sdepth(I_{n,m}^t:w)\leq \min\{ \sdepth(U_{m,t}S),\sdepth(P_{m,t,q}S) \}.$$
Using Proposition \ref{p41}, \cite[Lemma 3.6]{hvz} and \cite[Theorem 2.2]{biro} we deduce that
$$\sdepth(I_{n,m}^t:w)\leq \min\{ n-\left\lceil \frac{t}{2} \right\rceil, n-q+1\}$$
In the case $r=m$, we obtain the same inequality. Therefore, the required conclusion follows
from Lemma \ref{lima}(2) and the fact that $q=\left\lfloor \frac{n-t+1}{m+1} \right\rfloor$.
\end{proof}

\subsection*{The $m$-path ideal of a cycle graph}

Let $n > m\geq 2$ be two integer and
$$J_{n,m}=I_{n,m}+(x_{n-m+2}\cdots x_nx_1,\ldots,x_nx_1\cdots x_{m-1})\subset S=K[x_1,\ldots,x_n],$$
the $m$-path ideal associated to the cycle graph of length $n$.

Let $d:=\gcd(n,m)$. We consider the ideal
\begin{equation}
U'_{n,d}=(x_1,x_{d+1},\cdots,x_{d(r-1)+1})\cap (x_2,x_{d+2},\cdots,x_{d(r-1)+2})\cap \cdots \cap (x_d,x_{2d},\ldots,x_{rd}),
\end{equation}
where $r:=\frac{n}{d}$. Note that $U'_{n,1}=\mathbf m=(x_1,\ldots,x_n)$.

\begin{prop}\label{p42}
With the above notations, we have that:
$$\sdepth(U'_{n,d})\leq \qdepth(U'_{n,d})\leq \left\lfloor \frac{n+d}{2} \right\rfloor.$$
\end{prop}

\begin{proof}
According to Theorem \ref{t31}, we have that
$$\qdepth(U'_{n,d})\leq \left\lfloor \frac{n+d}{2} \right\rfloor.$$
Now, apply Proposition \ref{p1}.
\end{proof}



Let $t_0:=t_0(n,m)$ be the maximal integer such that $t_0\leq n-1$ and there exists a positive
integer $\alpha$ such that 
$$mt_0 = \alpha n + d.$$
Let $t\geq t_0$ be an integer.

\begin{teor}\label{t42}
Let $n > m\geq 2$ and $t\geq t_0$. We have that $$\sdepth(J_{n,m}^t)\leq \left\lfloor \frac{n+d}{2} \right\rfloor.$$
\end{teor}

\begin{proof}
 According to \cite[Lemma 2.2]{lucrare3}, there exists a monomial $w_t\in S$ such that
    $$(J_{n,m}^t:w_t)=U'_{n,d}.$$
		The conclusion follows from Lemma \ref{lima}(2) and Proposition \ref{p42}.
\end{proof}

\subsection*{Aknowledgments} 

We gratefully acknowledge the use of the computer algebra system Cocoa (\cite{cocoa}) for our experiments.

The second author, Mircea Cimpoea\c s, was supported by a grant of the Ministry of Research, Innovation and Digitization, CNCS - UEFISCDI, 
project number PN-III-P1-1.1-TE-2021-1633, within PNCDI III.





\end{document}